\documentclass{article}

\usepackage[margin=1.5in]{geometry}  
\usepackage{graphicx}   
\usepackage{fancyhdr}
\usepackage{amsmath}               
\usepackage{amsfonts}              
\usepackage{amsthm}                
\usepackage{enumitem}
\usepackage{amssymb}
\usepackage{hyperref}
\usepackage{tabularx}
\usepackage{amsmath}

\usepackage{tikz}
\usetikzlibrary{arrows}
\usetikzlibrary{matrix}

\tikzset{diagram/.style={matrix of math nodes, inner sep=0pt, row
    sep=#1, column sep=2.5em, text height=1.5ex, text depth=.25ex,
    nodes={inner sep=1ex}}}
\tikzset{diagram/.default=2.5em}

\newtheorem{thm}{Theorem}
\newtheorem{prop}[thm]{Proposition}
\newtheorem{lemma}[thm]{Lemma}
\newtheorem{cor}[thm]{Corollary}

\newtheorem*{thm*}{Theorem}

\theoremstyle{definition}

\newtheorem{rmk}{Remark}

\newtheorem*{related works}{Related Works}

\newtheorem*{question*}{Question}

\theoremstyle{definition}
\newtheorem{question}{Question}

\numberwithin{equation}{section}

\newcommand{\CC}{\mathbb{C}}
\newcommand{\RR}{\mathbb{R}}
 
\newcommand{\QQ}{\mathbb{Q}}

\newcommand{\A}{\mathbb{A}}
\newcommand{\ZZ}{\mathbb{Z}}

\newcommand{\s}{\sigma} 
\newcommand{\ld}{\lambda}

\newcommand{\fq}{\mathbb{F}_q}
\newcommand{\fqbar}{\overline{\ff}_q}
\newcommand{\ff}{\mathbb{F}}

\newcommand{\zt}{Z(V,t)}
\newcommand{\zto}{\mathring{Z}(V,t)}

\newcommand{\cvq}{\mathrm{Conf}_{n}V(\fq)}

\newcommand{\pv}{\mathrm{PConf}_{n}V}
\newcommand{\cv}{\mathrm{Conf}_{n}V}
\newcommand{\cnq}{\mathrm{Conf}_{n}(\fq)}

\newcommand{\cq}{\mathrm{Conf}_{n}(\fq)}
\newcommand{\cc}{\mathrm{Conf}_{n}(\CC)}

\newcommand{\pc}{\mathrm{PConf}_{n}(\CC)}
\newcommand{\cn}{\mathrm{Conf}_{n}}

\newcommand{\F}{\mathrm{Frob}_q}

\newcommand{\PP}{\mathbb{P}}

\newcommand{\cM}{\mathrm{Conf}_n M}
\newcommand{\pM}{\mathrm{PConf}_n M}

\newcommand{\tn}{\mathcal{T}_n}
\newcommand{\tq}{\mathcal{T}_n(\fq)}
\newcommand{\tc}{\mathcal{T}_n(\CC)}
\newcommand{\Tc}{\Tn(\CC)}

\newcommand{\Tn}{\widetilde{\mathcal{T}}_n}
\newcommand{\gl}{\mathrm{GL}_{n}}
\newcommand{\glq}{\mathrm{GL}_{n}(\fq)}

\AtEndDocument{\bigskip{\footnotesize\par
\textsc{Department of Mathematics,}\par 
\textsc{University of Chicago,}\par
\textsc{5734 S. University Ave.}\par
\textsc{Chicago, IL 60637, U.S.A.} \par 
  \addvspace{\medskipamount} 
  \text{E-mail}: 
\texttt{chen@math.uchicago.edu} \par
}}

\begin{document}

\nocite{*}

\title{Twisted cohomology of configuration spaces and spaces of maximal tori via point-counting}

\author{Weiyan Chen}
\date{\today}

\maketitle

\begin{abstract}
We consider two families of algebraic varieties $Y_n$ indexed by natural numbers $n$: the configuration space of  unordered $n$-tuples of distinct points on $\CC$, and the space of  unordered $n$-tuples of linearly independent lines in $\CC^n$. Let $W_n$ be any sequence of virtual $S_n$-representations given by a character polynomial, we compute $H^i(Y_n; W_n)$ for all $i$ and all $n$ in terms of double generating functions. One consequence of the  computation is a new recurrence phenomenon: the stable twisted Betti numbers $\lim_{n\to\infty}\dim H^i(Y_n; W_n)$ are  linearly recurrent in $i$. Our method is to compute twisted point-counts on the $\fq$-points of certain algebraic varieties, and then pass through the Grothendieck-Lefschetz fixed point formula to prove results in topology. We also generalize a result of Church-Ellenberg-Farb about the configuration spaces of the affine line to those of a general smooth variety.



\end{abstract}

\section{Introduction}

We consider two families of spaces indexed by natural numbers $n$. The first family is the configuration space of ordered $n$-tuples of distinct points in a manifold $M$:
\[\pM:=\{(x_1,\cdots,x_n)\in M^n: \ x_i\ne x_j,\ \ \forall i\ne j\}.\]
The symmetric group $S_n$ acts freely on $\pM$ by permuting the ordered points. The quotient $\cM:=\pM/S_n$ is the configuration space of \emph{unordered} $n$-tuples of distinct points. 
The second family is the space of $n$ linearly independent lines in $\CC^n$:
\[\Tn(\CC):=\{(L_1,\cdots, L_n): \text{$L_i$ a line in $\CC^n$,\ $L_1,\cdots, L_n$ linearly independent}\}.\]
$S_n$ acts freely on $\Tc$ by permuting the ordered lines. The quotient $\tc:=\Tc/S_n$ can be identified with the space of maximal tori in $\gl(\CC)$. See Section \ref{at tori} for more details.

Every normal $S_n$-cover $X\to Y$ gives a natural bijection between representations of $S_n$ and local systems on $Y$ that become trivial when restricted to $X$. Thus, every $S_n$-representations give rise to a local system on $\cM$ and on $\tc$.

\begin{question}[\bf{Twisted Betti numbers}]\label{q1}
What are the twisted Betti numbers $\dim H^i(\cM;W_n)$ and $\dim H^i(\tc;W_n)$ for each $i$ and $n$, and for each representation $W_n$ of $S_n$?
\end{question}

These twisted Betti numbers have geometric, arithmetic, and combinatorial meaning (see \emph{e.g.} Sections 2 and 5 in \cite{ICM}).  The program of computing these numbers dates back to the work of Arnol'd in the 1960s. For example, 
if $W_n$ is the trivial, the sign, or the standard representations of $S_n$, then $\dim H^i(\cc;W_n)$ have been known for all $i$ and $n$, by the work of Arnol'd \cite{Arnol'd}, Cohen \cite{Cohen}, and Vassiliev \cite{V}. 
However, even in the special case when $M=\CC$, there is no known formula of $\dim H^i(\cc;W_n)$,  for every $i$ and $n$ and $W_n$. In his 2014 ICM talk, Farb proposed a list of problems, one of which (Problem 2.1 in \cite{ICM}) is equivalent to Question \ref{q1}. See Remark \ref{icm} below for more details. 

This paper contains two collections of results: one topological and one arithmetic. We will use the arithmetic results to obtain results in topology. 
\\

\noindent\textbf{Topological results:}
\begin{itemize}
\item Theorem \ref{intro main} computes $\dim H^i(\cc;W_n)$ and $\dim H^i(\tc;W_n)$ for all $i$ and all $n$, and for all representations $W_n$ of $S_n$. This answers Question \ref{q1} for $M=\CC$.

\item In Corollary \ref{intro cor}, we discover a new recurrence phenomenon: the stable twisted Betti numbers $\lim_{n\to\infty}\dim H^i(\cc;W_n)$ and $\lim_{n\to\infty}\dim H^i(\tc;W_n)$ satisfy linear recurrence relations in $i$.
\end{itemize}
\noindent\textbf{Arithmetic results:}
\begin{itemize}
\item Theorem \ref{main} computes weighted point-counts on the $\fq$-points of $\cv$ where $V$ is a smooth variety. 

\item Corollary \ref{n limit} states that when $n\to\infty$, the weighted point-counts on the $\fq$-points of $\cv$ converges in some appropriate sense. This gives a new proof of a recent theorem of Farb-Wolfson and generalizes a theorem of Church-Ellenberg-Farb.
\end{itemize}

\subsection{Computing twisted Betti numbers.} 

We will consider Question \ref{q1} in a more general setting, where $W_n$ is allowed to be a \emph{virtual} $S_n$-representation, \emph{i.e.} a formal $\QQ$-linear combination of $S_n$-representations. Virtual representations are in natural bijection with the set of class functions of $S_n$. In this case, $\dim H^i(\cc; {W}_n)$ and $\dim H^i(\tn(\CC); {W}_n)$ are now well-defined  rational numbers since the cohomology functor is additive in coefficients. 

For each positive integer $k$, define $X_k:\coprod_{n=1}^\infty S_n\to\ZZ$ to be the class function with $X_k(\s)$ the number of $k$-cycles in the unique cycle decomposition of $\s\in S_n$. A \emph{character polynomial} is a polynomial $P\in\QQ[X_1,X_2,\cdots]$. It defines a class function on $S_n$ for all $n$. Define the \emph{degree} of a character polynomial by letting each variable $X_k$ to have degree $k$.  For   a sequence of nonnegative integers $\ld=(\ld_1,\cdots,\ld_l)$, define a character polynomial by
$${X\choose \ld}:={X_1\choose \ld_1}{X_2\choose \ld_2}\cdots{X_l\choose \ld_l}.$$
Then $X\choose\ld$ has degree $|\ld|:=\sum_{k=1}^l k\ld_k$.
For each fixed $n$, every class function on $S_n$ is a $\QQ$-linear combination of character polynomials of the form $X\choose \ld$. For example, the indicator function on the conjugacy class of $\s\in S_n$ is  $X\choose\ld$ where $\ld=(X_1(\s),\cdots,X_n(\s))$.
Therefore, to answer Question \ref{q1}, it suffices to consider the case $W_n:={X\choose\ld}$.

\begin{thm}[\textbf{Generating function for twisted Betti numbers}]
\label{intro main}
 Let $\mu$ be the classical M\"obius function, and let $M_k(z^{-1}):= \frac{1}{k}\sum_{j|k} \mu(\frac{k}{j})z^{-j}$ be the $k$-th necklace polynomial in $z^{-1}$. For any sequence of nonnegative integers $\ld=(\ld_1,\ld_2,\cdots,\ld_l)$, we have the following two equations of formal power series in two variables $z$ and $t$.
\begin{align*}
&\text{(I) }\ \ \ \ \sum_{n=0}^\infty\sum_{i=0}^\infty \dim H^i(\cc;{X\choose\ld})(-z)^it^n=\frac{1-zt^2}{1-t}\prod_{k=1}^l \dbinom{M_k(z^{-1})}{\ld_k}\bigg(\frac{(tz)^k}{1+(tz)^k}\bigg)^{\ld_k}\\
&\text{(II) }\ \ \ \ \sum_{n=0}^\infty\sum_{i=0}^\infty \frac{\dim H^{2i}(\tn(\CC);{X\choose\ld})}{(1-z)(1-z^2)\cdots(1-z^n)}z^it^n=\bigg[\prod_{k=1}^l\frac{1}{\ld_k!}\bigg(\frac{t^k}{k(1-z^k)}\bigg)^{\ld_k}\bigg]\cdot\prod_{j=0}^\infty\frac{1}{1-tz^j}
\end{align*} 
%
%
%
\end{thm}
In (I), all negative power of $z$ in $M_k(z^{-1})$ will cancel with other positive powers of $z$ so that the right-hand-side of the equality is indeed a series in $z$ and $t$. In (II), we only consider $H^{2i}(\tn(\CC);{X\choose\ld})$ because $H^{2i+1}(\tn(\CC);{X\choose\ld})=0$ by the work of Borel \cite{Borel}. 
%
%
%
%
%

\begin{rmk}[\textbf{Representation stability}]
\label{icm}
Farb proposed the following problem (Problem 2.1 in \cite{ICM}): for a manifold $M$, compute the decomposition of $H^i(\pM;\QQ)$ into a sum of irreducible representations of $S_n$. Remarkably, such a decomposition does not depend on $n$ when $n$ is sufficiently large. This result of \emph{representation stability} was first proved by Church-Farb \cite{CF} for $M=\CC$, and later by Church \cite{Church} for $M$ any connected orientable manifold of finite type (see also \cite{CEF2} for a different proof). Farb proposed a second problem (Problem 3.5 in \cite{ICM}) of computing the \emph{stable} decomposition of $H^i(\pM;\QQ)$ when $n$ is large. 
%
Note that for any $S_n$-representation $W_n$, the transfer isomorphism associated to the $S_n$-cover $\pM\to\cM$ gives:
 \begin{equation}
\label{transfer}
\dim H^i(\cM; {W}_n)=\langle H^i(\pM;\QQ),W_n \rangle_{S_n},
\end{equation} 
where $\langle U,V\rangle_{S_n}$ stands for the usual inner product of two $S_n$-representations $U$ and $V$. Hence, computing the multiplicities of $W_n$ in the decomposition of $H^i(\pc;\QQ)$ are equivalent to computing twisted Betti numbers of $\cM$ in $W_n$.

The simplest nontrivial case for Farb's two questions is when $M=\CC$. Theorem \ref{intro main} (I) reduces Farb's two questions in this case  to computing Taylor expansions of  rational functions. See Section \ref{examples} for more discussion and examples.

\end{rmk}

\begin{rmk}[\textbf{Twisted homological stability}]
Representation stability for $\pc$ implies twisted homological stability for $\cc$. Precisely, Church-Ellenberg-Farb (Theorems 1.9 in \cite{CEF2}) proved that for any character polynomial $P$ and for each fixed $i$, the twisted Betti numbers $\dim H^i(\cc;P)$ stabilize when $n$ is sufficiently large. Later, Hersh-Reiner gave a different proof of the stability of $\dim H^i(\cc;P)$ with an improved stable range in $n$ (Theorem 4.3 in \cite{HR}). We will give a third proof of this stability result in Corollary \ref{st} using Theorem \ref{intro main}. The implied stable range is a small improvement of that obtained by Hersh-Reiner, and is optimal (see Remark \ref{sharp} below). The three papers (\cite{CEF2}, \cite{HR} and the present one) land at the same result from three totally different points of views respectively: topological, combinatorial, and arithmetic. 
\end{rmk}

\paragraph{Linear recurrence of stable twisted Betti numbers in $\pmb i$.} 
Besides finding new proofs of homological stability, we discover a new phenomenon: the stable cohomology of $\cc$ and $\tc$  as $n\to\infty$ with twisted coefficients are linearly recurrent in $i$. 

\begin{cor}[\textbf{Linear recurrence of stable twisted Betti numbers}]
\label{intro cor}
Fix an arbitrary character polynomial $P\in \QQ[X_1,X_2,\cdots]$. Let $N=\deg P$.
\begin{enumerate}[label=(\Roman*)]
\item  For each $i$, denote 
$\alpha_i:=\lim_{n\to\infty}\dim H^i(\cc;P).$
There exist  integers $c_1,\cdots,c_N$  such that for all $i\ge N+2$, 
$$\alpha_i= c_1\alpha_{i-1}+c_2\alpha_{i-2}+\cdots+c_N\alpha_{i-N}.$$

\item  For each $i$, denote  
$\beta_i:=\lim_{n\to\infty}\dim H^{2i}(\tc;P).$
There exist integers $d_1,\cdots,d_N$   such that for all $i\ge N$,
$$\beta_i= d_1\beta_{i-1}+d_2\beta_{i-2}+\cdots+d_N\beta_{i-N}.$$
\end{enumerate}
\end{cor}

For example, if we let $\alpha_i:=\lim_{n\to\infty}\dim H^i(\cc;{\bigwedge}^2\QQ^{n-1})$ where $\QQ^{n-1}$ is the standard representation of $S_n$, then $\alpha_i$ satisfies the linear recurrence relation:
$$\alpha_i=2\alpha_{i-1}-2\alpha_{i-2}+2\alpha_{i-3}-\alpha_{i-4}.$$
See Section \ref{examples} for more details. 

\begin{rmk}[\bf Topological proof?]
We deduce Corollary \ref{intro cor} from Theorem \ref{intro main} by explicitly calculating the generating functions of $\alpha_i$ and $\beta_i$ as rational functions. The proof of Theorem \ref{intro main} uses point-counting, hence crucially depends on the fact that $\cc$ and $\tc$ are algebraic varieties. Is there any proof of Corollary \ref{intro cor} using only topology? Are there other examples of recurrent stable twisted Betti numbers in $i$?
\end{rmk}

\paragraph{Method: point-counting over finite fields.} 
The method in this article combines ideas from two beautiful papers: one by Church-Ellenberg-Farb \cite{CEF} and the other by Fulman \cite{Fulman}. Church-Ellenberg-Farb observed that there is a remarkable bridge, provided by the Grothendieck-Lefschetz fixed point theorem in \'etale cohomology, between cohomology in local coefficients (topology) and weighted point-counts on varieties  over  finite fields (arithmetic). Furthermore, they apply representation stability in topology to prove that certain weighed point-counts converge. Later, Fulman used a different method to improve the  arithmetic calculations stated in \cite{CEF} and obtained certain ``finite $n$" formulas. In this paper, we will systematically extend Fulman's calculations of weighted point-counts, and combine it with the approach of Church-Ellenberg-Farb but in the opposite direction: we use point-counting to compute cohomology. 

The idea of using point-counting to study the topology of configuration spaces dates back at least to the work of Lehrer-Kisin \cite{LK}, and is also used in Section 4.3 of \cite{CEF}. Our results are continuations of the  theme developed by  Lehrer-Kisin and Church-Ellenberg-Farb: structures in the cohomology (\emph{e.g.} stability and recurrence) are often reflected in the arithmetic of corresponding varieties, and \emph{vice versa}.

%

\subsection{Weighted point-counts on configuration spaces of smooth varieties.}

Fulman's method in \cite{Fulman} allows us to generalize a result of Church-Ellenberg-Farb as follows. Let $\cv$  be the configuration space of unordered $n$-tuples of distinct points on a smooth variety $V$ defined over $\ZZ$. When $V$ is the affine line, $\cn\A^1$ is just $\cn$ as discussed above\footnote{For brevity we will consistently use $\cn$ to abbreviate for $\cn\A^1$ throughout the paper.}. In general, every class function of $S_n$ gives a function $\cvq\to\QQ$, which can be viewed as a weighting (see Section \ref{set up} for more details). The following theorem computes weighted point-counts on $\cvq$ in terms of the zeta function $\zt$ of $V$ over $\fq$.


\begin{thm}[\bf Weighted point-counts on $\pmb{\cv}$]
\label{main}
Let $V$ be a smooth, connected variety over $\ZZ$ of positive dimension, and let $q$ be any odd prime power. Let $\mu$ be the M\"obius function, and define $M_k(V,q) := \frac{1}{k}\sum_{m|k}\mu(\frac{k}{m}) |V(\ff_{q^m})|$ for each $k$. 
For any sequence of nonnegative integers $\ld=(\ld_1,\cdots,\ld_l)$, we have the following equality of formal power series in $t$:
\begin{equation}
\label{thm eq}
\sum_{n=0}^{\infty} \Bigg[\sum_{C\in \cvq} \dbinom{X}{\ld}(\s_C) \Bigg] t^n= \frac{\zt}{Z(V,t^2)}\cdot\prod_{k=1}^l \dbinom{M_k(V,q)}{\ld_k}\Bigg(\frac{t^k}{1+t^k}\Bigg)^{\ld_k}
\end{equation}
\end{thm}

Thanks to Weil conjectures (proved by Dwork, Grothendieck, Deligne \emph{et al}.), $\zt$ is a rational function in $t$ with a simple pole at $t=q^{-\dim V}$, which is of the smallest absolute value among all other poles or zeros of $\zt$. By examining the location of poles in the generating sequence (\ref{thm eq}), we see that any point-count on $\cvq$ weighted by a character polynomial converges as $n\to\infty$ in the following sense.

\begin{cor}[\bf Convergence of weighted point-counts]
\label{n limit}
With the same assumptions as in Theorem \ref{main} and letting $d$ be the dimension of the variety $V$, we have:
\begin{enumerate}[label=(\alph*)]
\item Define $\zto$ to be the rational function $\zt\cdot(1-q^dt)$ in $t$. Then
\begin{equation}
\label{cor1}
\lim_{n\to\infty}\frac{1}{q^{nd}}\sum_{C\in\cvq} {X\choose\ld}(\s_C) = \frac{\mathring{Z}(V,q^{-d})}{Z(V,q^{-2d})}\prod_{k=1}^l \dbinom{M_k(V,q)}{\ld_k}\Bigg(\frac{1}{1+q^{kd}}\Bigg)^{\ld_k}
\end{equation}
In particular, for any character polynomial $P$ the following limit exists:
\begin{equation}\label{conv}
\lim_{n\to\infty}\frac{1}{q^{nd}}\sum_{C\in\cvq} {P}(\s_C).
\end{equation}
\item The expected value of $X\choose\ld$ as a random variable on $\cvq$ converges:
$$\lim_{n\to\infty}\frac{1}{|\cvq|}\sum_{C\in\cvq} {X\choose\ld}(\s_C)=\prod_{k=1}^l\dbinom{M_k(V,q)}{\ld_k}\Bigg(\frac{1}{1+q^{kd}}\Bigg)^{\ld_k}$$

\end{enumerate}
\end{cor}

\begin{rmk}[\bf Related works]
The convergence of (\ref{conv}) in the special case when $V=\A^1$ was first proved by Church-Ellenberg-Farb (Theorem 1 in \cite{CEF}). Part (a) generalizes their result to a general smooth variety.
It concurs with the recent work of Farb-Wolfson, where they extend the topological approach of \cite{CEF} and  gives a different formula for the left-hand-side of (\ref{cor1}) in terms of the \'etale cohomology of $\pv$(Theorem B in \cite{FW}). Our proof, inspired by the work of Fulman, is different from the topological approach in \cite{CEF} and \cite{FW}. 
We obtain not only the asymptotic formula as $n\to\infty$ (Corollary \ref{n limit}), but also a generating function for all $n$ (Theorem \ref{main}). 
\end{rmk}

\begin{rmk}[\bf Probabilistic interpretation and analogs in number theory]
\label{random variable}
Part (b) of Corollary \ref{n limit} has the following probabilistic intepretation: the functions $X_1,X_2,X_3, \cdots$, viewed as random variables on $\cvq$, tends to independent random variables with binomial distribution as $n\to\infty$. This is a geometric analog of the following fact in number theory: the $p$-adic orders, for $p$ any prime number, of a random integer chosen uniformly from  $\{1,2,\cdots,n\}$  tend to be independent random variables with geometric distributions as $n\to\infty$. More results about weighted point-counts on $\cvq$ (and other related spaces) motivated by this probabilisitic point of view will be presented in the forthcoming work of the  author \cite{next}.
\end{rmk}

\paragraph{Acknowledgment.}
The author thanks Sean Howe, Jenny Wilson, and Rita Jim\'enez Rolland for helpful conversations on the subject. The author is deeply grateful to his advisor Benson Farb, both for his continued support on the project and for his extensive comments on an earlier draft. 


\section{Cohomology of configuration spaces via point counting}
\label{arithmetic}
In this section, we will first prove Theorem \ref{main} and Corollary \ref{n limit} about weighted point-counts on $\cvq$. The main ideas of the proofs were already contained in Fulman's paper \cite{Fulman}, though he only proved the formulas in the special case when $V=\A^1$ and when $\ld=(0)$, $(1)$ and $(0,1)$. We systematically extend Fulman's result to all $V$ and all $\ld$, using some technical input from the Weil conjectures. We then apply the general formula in the case when $V=\A^1$ to prove part (I) of Theorem \ref{intro main} and of Corollary \ref{intro cor} about $\cc$.

\subsection{General set-up}
\label{set up}
Throughout this section, we will fix $V$ to be a smooth and connected variety over $\ZZ$ of dimension $d\ge1$. Define the \emph{configuration space} of $V$ to be the (scheme-theoretic) quotient
$$\cv:=\{(x_1\cdots,x_n)\in V^n : x_i\ne x_j,\  \forall i\ne j\}/S_n.$$
where $S_n$ acts on $V^n$ by permuting the coordinates. $\cv$ is  also a variety over $\ZZ$ (by \cite{Mumford}, page 66). So we can study its $\fq$-points $\cvq$. An element in $\cvq$ is a set of distinct points $C=\{x_1,\cdots,x_n\}\subseteq V(\fqbar)$ such that the Frobenius map $\F:V(\fqbar)\to V(\fqbar)$ preserves the set. The action of $\F$ on $C$ gives a permutation $\s_C\in S_n$, well-defined and unique up to conjugacy. Therefore, any class function $\chi$ of $S_n$ gives a well-defined function $\cvq\to\QQ$ by $C\mapsto \chi(\s_C)$.

\paragraph{Example: $V=\A^1$.} When $V$ is the affine line $\A^1$, we use $\cn$ to abbreviate for $\cn\A^1$. 
Elements $C\in\cnq$ are in bijection with monic, square-free, degree-$n$ polynomials in $\fq[x]$ via the map $$C=\{x_1,\cdots,x_n\}\mapsto f_C(x):=(x-x_1)\cdots(x-x_n).$$
Under this bijection, $X_k(\s_C)$, defined as the number of $k$-cycles in $\s_C$, equals to the number of degree-$k$ factors  in the irreducible factorization of $f_C(x)$ over $\fq$. 

\subsection{Proof of Theorem \ref{main}}

First we recall some basic facts about the zeta function of a variety $V$ over $\fq$:
$$\zt:=\prod_{x}({1-q^{\deg x}})^{-1}$$
where the product is taken over all closed points $x$ on $V$ over $\fq$. Weil conjectures give that $\zt$ is a rational function in $t$. Let $M_k(V,q)$ denote the number of closed points on $V$ of degree $k$, which is equivalently the number of orbits of $\F$ acting on $V(\fqbar)$ of size $k$. We have 
\begin{equation}
\label{euler}
\zt = \prod_{k=1}^\infty ({1-q^k})^{-M_k(V,q)}.
\end{equation}
Note that the fixed points of $\F$ on $V(\fqbar)$ are precisely $V(\fq)$. Similarly, for each $k$ we have 
$$|V(\ff_{q^k})|=\sum_{m|k}mM_m(V,q).$$
By M\"obius inversion, 
$$M_k(V,t)=\frac{1}{k}\sum_{m|k}\mu(\frac{k}{m})|V(\ff_{q^m})|.$$

\begin{proof}[Proof of Theorem \ref{main}]
Define a formal power series in  $x_1,\cdots,x_l$ and $t$:
\begin{equation}
\label{def F}
F(x_1,\cdots,x_l,t):=\sum_{n=0}^{\infty} \bigg[\sum_{C\in \cvq} x_1^{X_1(\s_C)}x_2^{X_2(\s_C)}\cdots x_l^{X_l(\s_C)} \bigg] t^n
\end{equation}
Recall that an element $C\in\cvq$ is just a subset of $V(\fqbar)$ of size $n$ that is preserved by $\F$. Thus, every $C\in\cvq$ can be decomposed uniquely into a disjoint union of distinct orbits of $\F$ acting on $V(\fqbar)$. The number of $\F$-orbits in $C$ of size $k$ is $X_k(\s_c)$.
 The unique decomposition of $C\in \cvq$ into disjoint union of distinct $\F$-orbits gives the following product formula\footnote{This is analogous to how unique factorization for integers gives the Euler product formula of  Riemann zeta function.}. 
\begin{align}
\nonumber F(x_1,\cdots,x_l,t)&=\Bigg[\prod_{k>l} (1+t^k)^{M_k(V,q)}\Bigg]\prod_{k\leq l} ({1+x_kt^k})^{M_k(V,q)}\\
\nonumber&=\Bigg[\prod_{k=1}^\infty (1+t^k)^{M_k(V,q)}\Bigg]\prod_{k\leq l}\Bigg(\frac{1+x_kt^k}{1+t^k}\Bigg)^{M_k(V,q)}\\
\nonumber&=\Bigg[\prod_{k=1}^\infty \bigg(\frac{1-t^{2k}}{1-t^{k}}\bigg)^{M_k(V,q)}\Bigg]\prod_{k\leq l}\Bigg(\frac{1+x_kt^k}{1+t^k}\Bigg)^{M_k(V,q)}
\end{align}
By the product formula (\ref{euler}), we obtain 
\begin{equation}
\label{Euler}
F(x_1,\cdots,x_l,t) = \frac{\zt}{Z(V,t^2)}\prod_{k\leq l} \Bigg(\frac{1+x_kt^k}{1+t^k}\Bigg)^{M_k(V,q)}
\end{equation}

Next we apply the formal differential operator $$(\frac{\partial}{\partial x})^{\ld}:=(\frac{\partial}{\partial x_1})^{\ld_1}(\frac{\partial}{\partial x_2})^{\ld_2}\cdots(\frac{\partial}{\partial x_l})^{\ld_l}$$
 to the series $F(x_1,\cdots,x_l,t)$ and then evaluate at $(x_1,\cdots,x_l)=(1,\cdots,1)$,  obtaining the following equalities. The symbol $\ld!$ is an abbreviation for $(\ld_1!)(\ld_2!)\cdots(\ld_l!)$. Differentiating (\ref{def F}) gives
$$(\frac{\partial}{\partial x})^{\ld}F(1,\cdots,1,t)= \ld!\cdot\sum_{n=0}^{\infty} \Bigg(\sum_{C\in \cvq} \dbinom{X}{\ld}(\s_C) \Bigg) t^n$$
Differentiating (\ref{Euler}) gives
$$(\frac{\partial}{\partial x})^{\ld}F(1,\cdots,1,t)
=\ld!\cdot\frac{\zt}{Z(V,t^2)}\cdot\prod_{k=1}^l \dbinom{M_k(V,q)}{\ld_k}\Bigg(\frac{t^k}{1+t^k}\Bigg)^{\ld_k}$$
Theorem \ref{main} follows by equating these two expressions for $(\frac{\partial}{\partial x})^{\ld}F(1,\cdots,1,t)$. 
\end{proof}

\subsection{Proof of Corollary \ref{n limit}}
First we recall the following basic fact from calculus.
\begin{lemma}\label{Taylor}
Given $A(t)=\sum_{n=0}^\infty a_nt^n$ where $a_n$ are real numbers. Suppose $A(t)=H(t)/(1-ct)$ where $c$ is a constant, and the radius of convergence of $H(t)$ is strictly greater than $|c^{-1}|$. Then $\lim_{n\to\infty} \frac{a_n}{c^n}$ exists and is equal to $H(c^{-1})$.
\end{lemma}
Let
$$A(t):=\frac{\zt}{Z(V,t^2)}\cdot\prod_{k=1}^l \dbinom{M_k(V,q)}{\ld_k}\Bigg(\frac{t^k}{1+t^k}\Bigg)^{\ld_k} 
$$
The Riemann Hypothesis over finite fields (proved by Deligne \cite{D}) says that $\zt$ has a simple pole at $t=q^{-d}$ where $d=\dim V$. Moreover, each other zero or pole of $\zt$ has absolute value $q^{-j}$ for some $j\le 2d-1$. Thus,  
$$\zto:=\zt(1-q^dt)$$ has no pole at $|t|<q^{-d+\frac{1}{2}}$; while $1/Z(V;t^2)$ has no pole at $|t|<q^{-2d}<q^{-d}$ (recall that $d=\dim V>0$). Hence $A(t)$ and $c=q^{-d}$ satisfy 
the hypothesis of Lemma \ref{Taylor}, by which we conclude
$$\lim_{n\to\infty}\frac{1}{q^{nd}}\sum_{C\in \cvq} \dbinom{X}{\ld}(\s_C) = \bigg[A(t)(1-q^dt) \bigg]_{t=q^{-d}}$$
This establishes (\ref{cor1}).

Every character polynomial $P$ is a $\QQ$-linear combination of $X\choose\ld$ for different $\ld$. Thus, the limit (\ref{conv}) converges for all $P$. Part (a) is proved.

In the case when $\ld=(0)$, part (a) gives
\begin{equation}
\label{con}
\lim_{n\to\infty}\frac{|\cvq|}{q^{nd}}=\frac{\mathring{Z}(V,q^{-d})}{Z(V,q^{-2d})}.\end{equation}
Part (b) follows by taking the ratio of (\ref{cor1}) and (\ref{con}). 

\qed

\subsection{Connecting arithmetic and topology of $\cn$}
For the rest of this paper, we will focus on the case when $V=\A^1$. Recall that we use $\cn$ to abbreviate for $\cn\A^1$. Let $W$ be a representation of $S_n$, with character $\chi_W$. Church-Ellenberg-Farb  proved the following equation connecting arithmetic of $\cq$ and topology of $\cc$: (Proposition 4.1 in \cite{CEF})
\begin{equation}
\label{GL intro}
\sum_{C\in \cnq}\chi_{W}(\s_C) = q^n\sum_i(-1)^i \dim H^i(\cc;{W})\ q^{-i}.
\end{equation}
By additivity, same formula holds if we replace $W$ by a virtual representation. See Section 4 in \cite{CEF} for how  (\ref{GL intro}) is obtained from the Grothendieck-Lefschetz fixed point theorem in \'etale cohomology. Results from the previous section (in the case when $V=\A^1$) give us access to the left-hand-side of (\ref{GL intro}), from which we can prove results about $H^i(\cc;{W})$. 

\subsection{Proof of Theorem \ref{intro main}, (I)}

We abbreviate the twisted Betti number as 
\begin{equation}
\label{b_i(n)}\alpha_i(n):=\dim H^i(\cc;{X\choose \ld})
\end{equation}
for each $i$ and $n$. Define the double generating function for $\alpha_i(n)$ as the formal power series in two variables $z$ and $t$
\begin{equation}
\label{Phi}
\Phi_\ld(z,t):=\sum_{n=0}^\infty\sum_{i=0}^\infty \alpha_i(n)(-z)^it^n
\end{equation}
We want to compute $\Phi_\ld(z,t)$ as a rational function. We will need the following lemma.

\begin{lemma}\label{cal 2}
Suppose $\Phi(z,t)$ and $\Psi(z,t)$ are two power series in two formal variables $z$ and $t$. If for every prime power $q$, we have $\Phi(q^{-1},t)=\Psi(q^{-1},t)$ as formal power series in $t$, then $\Phi(z,t)=\Psi(z,t)$ as formal series in $z$ and $t$. 
\end{lemma}
\begin{proof}[Proof of Lemma]
Suppose $\Phi_\ld(t,z)=\sum_{n=0}^\infty \phi_n(z)t^n$ and $\Psi(t,z)=\sum_{n=0}^\infty \psi_n(z)t^n$, where $\phi_n(z)$ and $\psi_n(z)$ are formal series in $z$ for each $n$. By hypothesis, for 
every prime power $q$, we have $\phi_n(q^{-1})=\psi_n(q^{-1})$. Recall the following fact from calculus: 
\begin{itemize}
\item If an infinite series $h(z) = \sum_{i=0}a_i z^i$ converges at $z=z_0$, then it converges absolutely at all $z$ with $|z|<|z_0|$. 
\end{itemize}
Hence, both $\phi_n(z)$ and $\psi_n(z)$ are holomorphic functions on a disk with a positive radius centered at 0. Since $\phi_n(z)=\psi_n(z)$ for all $z\in\{q^{-1}\ |\ \text{$q$ is a prime power}\}$ which accumulates at $0$, it must be $\phi_n(z)=\psi_n(z)$ as holomorphic functions. By the uniqueness of power series expansion,  $\phi_n(z)=\psi_n(z)$ as formal series in $z$. Thus $\Phi(z,t)=\Psi(z,t)$ as formal series in $z$ and $t$.
\end{proof}

Now we evaluate the double generating function $\Phi_\ld(z,t)$ at $z=q^{-1}$.
\begin{align*}
\nonumber\Phi_\ld(q^{-1},t)&=\sum_{n=0}^\infty\sum_{i=0}^\infty(-1)^ib_i(n)q^{-i}t^n\\
\nonumber&=\sum_{n=0}^\infty \bigg[\sum_{C\in \cnq}{X\choose\ld}(\s_C) \bigg] (q^{-1}t)^n&\text{By (\ref{GL intro})}\\
&=\frac{Z(\A^1,tq^{-1})}{Z(\A^1,(tq^{-1})^2)}\prod_{k=1}^n \dbinom{M_k(\A^1,q)}{\ld_k}\Bigg(\frac{(tq^{-1})^k}{1+(tq^{-1})^k}\Bigg)^{\ld_k}&\text{By Theorem \ref{main}}
\end{align*}
The \emph{$k$-th necklace polynomial in $x$} is $$M_k(x)=\frac{1}{k}\sum_{m|k}\mu(\frac{k}{m})x^{m}.$$ 
A standard calculation gives that $Z(\A^1,t)=\frac{1}{1-qt}$, and that $M_k(\A^1,q)=M_k(q)$. Thus, we simplify the above:
\begin{equation}
\label{q series}
\Phi_\ld(q^{-1},t)=\frac{1-t^2q^{-1}}{1-t}\prod_{k=1}^n \dbinom{M_k(q)}{\ld_k}\Bigg(\frac{(tq^{-1})^k}{1+(tq^{-1})^k}\Bigg)^{\ld_k}
\end{equation}
Since (\ref{q series}) holds at $z=q^{-1}$ for any prime power $q$. By Lemma \ref{cal 2}, the same equation holds when $q^{-1}$ is replaced by a formal variable $z$. 

\qed

\subsection{Stability of Betti numbers}

It was known by the general theory of representation stability developed by Church-Ellenberg-Farb that for any character polynomial $P$, the twisted Betti numbers $\dim  H^i(\cc;P)$ will be independent of $n$ when $n$ is sufficiently large. We will give a different proof of this result with an improved stability range for $n$.

\begin{cor}\label{st}
For every character polynomial $P$ and for every $i$, we have 
\begin{equation}
\label{stable}
\dim  H^i(\cc;P)=\dim  H^i(\mathrm{Conf}_{n+1}(\CC);P)
\end{equation}
when $n\geq i+\deg P+1$. 
\end{cor}
\begin{rmk}\label{sharp}
Church-Ellenberg-Farb first proved (\ref{stable}) when $n\geq 2i+\deg P$ (Theorem 1 \cite{CEF}). Later, Hersh-Reiner gave a different proof of (\ref{stable}) with a better stable range: $n\ge\max\{2\deg P,\ \deg P+i+1\}$ (Theorem 4.3 in \cite{HR}). The stable range in Corollary \ref{st} is a small improvement of the range obtained by Hersh-Reiner, and is sharp, as we will show it in Section \ref{examples}.
\end{rmk}

\begin{proof}
It suffices to consider when $P={X\choose\ld}$ for some sequence $\ld=(\ld_1,\cdots,\ld_k)$. In this case $\deg {X\choose\ld}=\sum_{k}k\ld_k$. Let $\alpha_i(n)$ be as in (\ref{b_i(n)}), and let $\Phi_\ld(z,t)$ be as in (\ref{Phi}).
$$(1-t)\Phi_\ld(t,z)=
1+t\sum_{n=0}^\infty\sum_{i=0}^\infty [\alpha_i(n+1)-\alpha_i(n)]z^it^n$$
It suffices to check that $(1-t)\Phi_\ld(t,z)$ is a sum of monomials of the form $z^{i}t^{n}$ where $n-i\leq \sum_{k=0}^{l}k\ld_k+1$. 

We will say an infinite series in $z$ and $t$ \emph{has slope $\leq m$} if it is a sum of monomials $z^{i}t^{n}$ where $n-i\leq m$. We want to show that the series  given by Theorem \ref{intro main}
\begin{equation}
\label{expression}
(1-t)\Phi_\ld(t,z)= (1-zt^2)\prod_{k=1}^l \dbinom{M_k(z^{-1})}{\ld_k}\Bigg(\frac{(tz)^k}{1+(z)^k}\Bigg)^{\ld_k}
\end{equation}
 has slope $\leq \sum_{k=0}^{l}k\ld_k+1$. We analyze each factor.
\begin{itemize}
\item $(1-zt^2)$ has slope $\leq 2-1=1$. 
\item For each $k$, the factor $M_k(z^{-1})$ has slope $\leq k$. Thus, $\dbinom{M_k(z^{-1})}{\ld_k}$ has slope $\leq k\ld_k$. 
\item For each $k$, $\bigg[\frac{(tz)^k}{1+(tz)^k}\bigg]^{\ld_k}$ has slope $\leq0$. 
\end{itemize}
Therefore, the product in (\ref{expression}) has slope $\leq 1+\sum_{k=0}^{l}k\ld_k$. This establishes the corollary.
\end{proof}
\subsection{Proof of Corollary \ref{intro cor}, (I).}
Let $\alpha_i$ be $\alpha_i(n)$ when $n\ge i+|\ld|+1$ in the stable range. Define the generating function
$$\Phi_\ld^\infty(z):=\sum_{i=0}^\infty \alpha_i(-z)^i$$
By Lemma \ref{Taylor}, we can calculate $\Phi_\ld^\infty(z)$ using $\Phi_\ld(z,t)$:
\begin{align}
\nonumber\Phi_\ld^\infty(z) &= \bigg[(1-t)\Phi_\ld(z,t)\bigg]_{t=1}&\text{by Theorem \ref{intro main}}\\
\label{gen fun stable} &= {(1-z)}\prod_{k=1}^l \dbinom{M_k(z^{-1})}{\ld_k}\Bigg(\frac{z^k}{1+z^k}\Bigg)^{\ld_k}
\end{align}
In particular,  $\Phi_\ld^\infty(z)$ in (\ref{gen fun stable}) is a rational function in $z$. The denominator is a polynomial in $z$ of degree $\sum_{k=1}^{l}k\ld_k=|\ld|$. 
The numerator has degree at most $1+|\ld|$.
This implies that $\alpha_i$ satisfies a linear recurrence relation of length $|\ld|$ once $i>|\ld|+1$. 

\qed
%
%

\subsection{Examples}
\label{examples}

Recall that  irreducible representations of $S_n$ are in bijection with partitions of $n$. For a fixed partition $\mu\vdash n$ where $\mu=(\mu_1\ge\mu_2\ge\cdots\ge\mu_r)$, we will denote by $V(\mu)_n$ \footnote{Sometimes we will suppress $n$ from the notation.} the representation of $S_n$ corresponding to the partition $n=(n-\sum_{i=1}^r \mu_i)+\mu_1+\cdots+\mu_r$ for all $n$ sufficiently large, \emph{i.e.} for $n-\sum_{i=1}^r \mu_i\geq \mu_1$. Going from $V(\mu)_n$ to $V(\mu)_{n+1}$ corresponds to adding one block in the first row of the corresponding Young diagram. 
Church-Ellenberg-Farb proved that $H^i(\pc;\QQ)$ is \emph{multiplicity stable} (Theorem 1.9 in \cite{CEF2}): for each  $i$, there is a finite set $Q_i$ of partitions such that
\begin{equation*}
H^i(\pc;\QQ)\cong\bigoplus_{\mu\in Q_i} V(\mu)_n^{\oplus d_i(\mu)}
\end{equation*}
for all $n$ sufficiently large. In particular, the sum over $Q_i$ is independent of $n$. Farb proposed the problem of computing $d_i(\mu)$ for each $i$ and each $\mu$ (Problem 3.5 in \cite{ICM}). Macdonald proved that for all partition $\mu$, the character of $V(\mu)_n$ is given by a unique character polynomial $P_\mu$ for all $n$ sufficiently large (Example I.7.14 in \cite{Mac}). Therefore, by transfer (\ref{transfer}), computing $d_i(\mu)$ is equivalent to computing the stable cohomology of $H^i(\cc;P_\mu)$. We will demonstrate the case of computing these using Theorem \ref{intro main} (I) in three examples where $\mu$ is the partition $1=1$, or $2=1+1$, or $2=2$.\\


\noindent {{\textbf{Example 1: $\pmb {W_n=V(1)_n}$}.}
Assume $n\geq2$, the irreducible representation $V(1)_n$ corresponds to the Young diagram $(n-1,1)$. 
It is also known as the \emph{standard representation}:
 $$V(1)_n \cong\{(x_1,\cdots,x_n)|\sum x_i=0\}\cong\QQ^{n-1}$$
where $S_n$ acts by permuting the coordinates. 

The $S_n$-character of $W$ is given by the character polynomial $X_1-1$. If we abbreviate the Betti number as
$$\alpha_i(n)=\dim H^i(\cc;V(1)_n),$$
then Theorem \ref{intro main} gives that the double generating function of $\alpha_i(n)$ is
\begin{align*}
\sum_{n=0}^\infty\sum_{i=0}^\infty& (-1)^i\alpha_i(n)\ z^it^n  = \frac{1-t^2z}{1-t}\bigg[\frac{t}{1+tz}-1\bigg]\\
&=(-z+z^2)t^3+(-z+2z^2-z^3)t^4 +(-z+2z^2-2z^3+z^4)t^5\\
&\ \ \ \ +(-z+2z^2-2z^3+2z^4+z^5)t^6+\cdots
\end{align*}
Thus, we conclude that when $n\geq 3$,
\begin{equation*}
\alpha_i(n)=
  \begin{cases}
   0 \ \ \ &  i=0 \\
   1 \ \ \ &  i=1 \\
   2 \ \ \ &  0<i<n-1\\
   1 \ \ \ &  i=n-1\\
  \end{cases}
\end{equation*}
\begin{rmk}
A computation of $\dim H^i(\cc;V(1)_n)$ from Lehrer-Solomon's description of $H^i(\pc;\QQ)$ was presented in Proposition 4.5 of \cite{CEF}. It took about one and half pages.  The computation above using generating function is a faster procedure.
\end{rmk}
The stable Betti numbers are:
\begin{equation*}
\alpha_i:=\lim_{n\to\infty}\dim H^i(\cc;V(1))=
  \begin{cases}
   0 \ \ \ &  i=0 \\
   1 \ \ \ &  i=1 \\
   2 \ \ \ &  i>1\\
  \end{cases}
\end{equation*}
When $i\geq 2$, the stable Betti numbers $\alpha_i$ are the same, which in particular satisfy a recurrence relation of length 1. From this example we see that the bounds in Corollary \ref{intro cor} (I) and Corollary \ref{st} are sharp.\\

\noindent {{\textbf{Example 2: $\pmb{W_n=V(1,1)_n.}$}} Assume $n\geq3$, the irreducible representation $V(1,1)_n$ corresponds to the Young diagram $(n-2,1,1)$. The dimension of $V(1,1)$ is $(n^2-3n+2)/{2}$. In fact, we have $V(1,1)\cong{\bigwedge}^2\QQ^{n-1}$ where $\QQ^{n-1}$ is the standard representation $V(1)$.  The character of $V(1,1)$ is given by the following character polynomial:
$${X_1\choose2}-X_1-X_2+1$$
If we abbreviate the Betti numbers $\alpha_i(n)=\dim H^i(\cc;V(1,1)_n),$ then
Theorem \ref{intro main} gives that
\begin{align*}
\sum_{n=0}^\infty\sum_{i=0}^\infty& (-1)^i\alpha_i(n)\ z^it^n\ =\ \Phi_{(2)}(z,t)-\Phi_{(1)}(z,t)-\Phi_{(0,1)}(z,t)+\Phi_{(0)}(z,t)\\
&=\frac{1-t^2z}{1-t}\bigg[\frac{(1-z)t^2}{2(1+tz)^2}-\frac{t}{1+tz}-\frac{(1-z)t^2}{2(1+(tz)^2)}+1\bigg]
\end{align*}

By expanding the generating function, we have the following table of the Betti numbers:\begin{center}
\begin{tabular}{|r||c|c|c|c|c|c|c|c|c|c|c|c|}
\hline
\multicolumn{13}{|c|}{$\alpha_i(n)=\dim H^i(\cc;V(1,1)_n)$}                                                         \\ \hline\hline
  $(i,n)$    & $n=3$ & 4 & 5 & \multicolumn{1}{c|}{6} & 7 & 8 & 9  & 10 & 11 & 12 & \multicolumn{1}{c|}{13} & 14 \\ \hline\hline
$i=0$ & \textbf{0} & 0 & 0 & 0                      & 0 & 0 & 0  & 0  & 0  & 0  & 0                       & 0  \\ \hline
1     & 0 & \textbf{0} & 0 & 0                      & 0 & 0 & 0  & 0  & 0  & 0  & 0                       & 0  \\ \hline
2     & 0 & 1 & \textbf{2} & 2                      & 2 & 2 & 2  & 2  & 2  & 2  & 2                       & 2  \\ \hline
3     &   & 1 & 3 & \textbf{5}                      & 5 & 5 & 5  & 5  & 5  & 5  & 5                       & 5  \\ \hline
4     &   &   & 1 & 4                      & \textbf{6} & 6 & 6  & 6  & 6  & 6  & 6                       & 6  \\ \hline
5     &   &   &   & 1                      & 5 & \textbf{7} & 7  & 7  & 7  & 7  & 7                       & 7  \\ \hline
6     &   &   &   &                        & 2 & 7 & \textbf{10} & 10 & 10 & 10 & 10                      & 10 \\ \hline
7     &   &   &   &                        &   & 3 & 9  & \textbf{13} & 13 & 13 & 13                      & 13 \\ \hline
8     &   &   &   &                        &   &   & 3  & 10 & \textbf{14} & 14 & 14                      & 14 \\ \hline
9     &   &   &   &                        &   &   &    & 3  & 11 & \textbf{15} & 15                      & 15 \\ \hline
10    &   &   &   &                        &   &   &    &    & 4  & 13 & \textbf{18}                      & 18 \\ \hline
11    &   &   &   &                        &   &   &    &    &    & 5  & 15                      & \textbf{21} \\ \hline
12    &   &   &   &                        &   &   &    &    &    &    & 5                       & 16 \\ \hline
13    &   &   &   &                        &   &   &    &    &    &    &                         & 5  \\ \hline
\end{tabular}
\end{center}

The bold entries lie on the line $n=i+3$. In each row, the Betti number stabilizes as $n\geq i+3$. This agrees with the stability bound as predicted in Corollary \ref{stable}: $n>i+\deg({X_1\choose2}-X_1-X_2+1)=i+2$. Moreover, we can see from the table that the bound is sharp.

Furthermore, from (\ref{gen fun stable}), we have the following formula for the generating function of the stable Betti numbers $\alpha_i:=\lim_{n\to\infty}\dim H^i(\cc;V(1,1)_n)$:
\begin{align*}
\sum_{i=0}^\infty (-1)^i\alpha_iz^i&= \Phi_{(2)}^\infty(z)-\Phi_{(1)}^\infty(z)-\Phi_{(0,1)}^\infty(z)
+\Phi_{(0)}^\infty(z)=(1-z)\bigg[\frac{1-z}{2(1+z)^2}-\frac{1}{1+z}-\frac{1-z}{2(1+z^2)}+1\bigg]\\
&=2z^2-5z^3+6z^4-7z^5+10z^6-13z^7+14z^8-15z^9+18z^{10}-21z^{11}+\cdots
\end{align*}
The stable Betti numbers satisfy the linear recurrence relation:
$$\alpha_i=2\alpha_{i-1}-2\alpha_{i-2}+2\alpha_{i-3}-\alpha_{i-4}.$$
By explicitly solving the recurrence relation, we have $\alpha_0=0$, $\alpha_1=2$, and when $i\geq 3$, 
\begin{equation*}
\alpha_i=
  \begin{cases}
   2i-2 \ \ \ &  i=0 \mod 4\\
   2i-3 \ \ \ &  i=1\mod 4\\
   2i-2 \ \ \ &  i= 2\mod 4\\
   2i-1 \ \ \ &  i= 3\mod 4\\
  \end{cases}
\end{equation*}
\begin{rmk}
In Section 4.4 of \cite{CEF}, Church-Ellenberg-Farb used $L$-functions to compute the stable cohomology of $H^i(\cc;{\bigwedge}^2\QQ^{n})$. Since ${\bigwedge}^2\QQ^{n}\cong {\bigwedge}^2\QQ^{n-1}\oplus\QQ^{n}$, we recover their computation. Moreover, we also obtained unstable cohomology.
\end{rmk}

\noindent {{\pmb{Example 3: $W_n=V(2)_n.$}}
Assume $n\geq4$, the irreducible representation $V(2)_n$ corresponds to the Young diagram $(n-2,2)$. 
The dimension of $V(2)$ is $(n^2-3n)/2$. In fact, $V(2)$ is a direct summand in the symmetric square of the standard representation $\QQ^{n-1}$. More precisely, we have
$$\mathrm{Sym}^2(\QQ^{n-1}) \cong \QQ^{n}\oplus V(2)$$
The character of $V(2)$ is given by the following character polynomial
$${X_1\choose2}+X_2-X_1$$
If we abbreviate the Betti numbers
$\alpha_i(n):=\dim  H^i(\cc;V(2)_n)$,
Theorem \ref{intro main} gives 
\begin{align*}
\sum_{n=0}^\infty\sum_{i=0}^\infty (-1)^i\alpha_i(n)\ z^it^n  &=\Phi_{(2)}(z,t)+\Phi_{(0,1)}(z,t)-\Phi_{(1)}(z,t)\\
&= \frac{1-t^2z}{1-t}\bigg[\frac{(1-z)t^2}{2(1+tz)^2}+\frac{(1-z)t^2}{2(1+{(tz)}^2)}-\frac{t}{1+tz}\bigg]
\end{align*}

By expanding the generating function, we have the following table of  Betti numbers:%
\begin{center}
\centering
\begin{tabular}{|r||c|c|c|c|c|c|c|c|c|c|c|}
\hline
\multicolumn{12}{|c|}{$\alpha_i(n)=\dim H^i(\cc;V(2)_n)$}                                                                                      \\ \hline\hline
 $(i,n)$    & $n=4$ & 5 & \multicolumn{1}{c|}{6} & 7 & 8 & 9 & 10 & 11 & 12 & \multicolumn{1}{c|}{13} & 14 \\ \hline\hline
$i=0$          & 0   & 0   & 0                        & 0   & 0   & 0   & 0    & 0    & 0    & 0                         & 0    \\ \hline
1       & \textbf{1}   & 1   & 1                        & 1   & 1   & 1   & 1    & 1    & 1    & 1                         & 1    \\ \hline
2  &      1   & \textbf{2}   & 2                        & 2   & 2   & 2   & 2    & 2    & 2    & 2                         & 2    \\ \hline
3  &         0   & 2   &  \textbf{3}                       &   3 & {3}   & 3   & 3    & 3    & 3    & 3                         & 3    \\ \hline
4  &         & 1   & 4                        & \textbf{6}   & 6   & 6   & 6    & 6    & 6    & 6                         & 6    \\ \hline
5  &           &     & 2                        & {6}   & \textbf{9}   & 9   & 9    & 9    & 9    & 9                         & 9    \\ \hline
6         &     &     &                          & 2   & 7   & \textbf{10}  & 10   & 10   & 10   & 10                        & 10   \\ \hline
7        &     &     &                          &     & 2   & 8   & \textbf{11}   & 11   & 11   & 11                        & 11   \\ \hline
8         &     &     &                          &     &     & 3   & 10   & \textbf{14}   & 14   & 14                        & 14   \\ \hline
9      &     &     &                          &     &     &     & 4    & 12   & \textbf{17}   & 17                        & 17   \\ \hline
10        &     &     &                          &     &     &     &      & 4    & 13   & \textbf{18}                        & 18   \\ \hline
11      &     &     &                          &     &     &     &      &      & 4    & 14                        & \textbf{19}   \\ \hline
12 &   &     &                          &     &     &     &      &      &      & 5                         & 16   \\ \hline
13        &     &     &                          &     &     &     &      &      &      &                           & 6    \\ \hline
\end{tabular}
\end{center}

The bold entries lie on the line $n=i+3$. In each row, the Betti number stabilizes when $n\geq i+3$. This agrees with the stability bound as predicted in Corollary \ref{stable}: $n>i+\deg({X_1\choose2}+X_2-X_1)=i+2$. We can see from the table that the bound is sharp. 

Furthermore, from (\ref{gen fun stable}), we have the following formula for the generating function of the stable Betti numbers:
\begin{align*}
\sum_{i=0}^\infty (-1)^i\alpha_iz^i&= \Phi_{(2)}^\infty(z)-\Phi_{(1)}^\infty(z)-\Phi_{(0,1)}^\infty(z)
+\Phi_{(0)}^\infty(z)=(1-z)\bigg[\frac{1-z}{2(1+z)^2}+\frac{1-z}{2(1+z^2)}-\frac{1}{1+z}\bigg]\\
&=-z+2z^2-3z^3+6z^4-9z^5+10z^6-11z^7+14z^8-17z^9+18z^{10}-19z^{11}+\cdots
\end{align*}
The stable Betti numbers satisfies the linear recurrence relation:
$$\alpha_i=2\alpha_{i-1}-2\alpha_{i-2}+2\alpha_{i-3}-\alpha_{i-4}.$$
We can explicitly solve the recurrence relation and obtain that $\alpha_0=0$, and when $i\ge1$, 
\begin{equation*}
\alpha_i:=\lim_{n\to\infty}\dim H^i(\cc;V(2))=
  \begin{cases}
   2i-2 \ \ \ &  i=0 \mod 4\\
   2i-1 \ \ \ &  i=1\mod 4\\
   2i-2 \ \ \ &  i= 2\mod 4\\
   2i-3 \ \ \ &  i= 3\mod 4\\
  \end{cases}
\end{equation*}

\section{Cohomology of $\tc$ via point counting}
\label{at tori}

In this section we prove part (II) of Theorem \ref{intro main} and Corollary \ref{intro cor}. Our analysis of $\tn$ closely parallels that of $\cn$ before.

\subsection{General set-up}
%

$\tn=\Tn/S_n$ is a scheme over $\ZZ$ (again, see page 66 in \cite{Mumford}). The $\fq$-points $\tq$ consists of sets $L=\{L_1,\cdots,L_n\}$ of $n$ linearly independent lines in $\PP^{n-1}(\fqbar)$ 
such that the Frobenius map $\F:\PP^{n-1}(\fqbar)\to \PP^{n-1}(\fqbar)$ preserves the set $T$. 

Let $F$ abbreviate the Frobenius map. An $F$-stable \emph{torus} in $\glq$ is an algebraic subgroup which becomes diagonalizable over $\fqbar$. An  $F$-stable torus is \emph{maximal} if it is not properly contained in any larger one. Given any $F$-stable maximal torus $T$, its $n$ eigenvectors in $\fqbar^n$ defines a set $L_T$ of $n$ independent lines in $\fqbar^n$. Thus $L_T$ is a element of $\tq$. The map $T\mapsto L_T$ gives a bijection between $F$-stable maximal tori in $\glq$ and $\fq$-points of $\tn$. Therefore, $\tq$ is precisely the set of $F$-stable maximal tori in $\glq$. See Section 5.1 of \cite{CEF} for a proof.

For any $T\in\tq$, the action of $\F$ on $L_T$, a set of $n$ lines in $\fqbar^n$, gives a permutation $\s_T\in S_n$, unique up to conjugacy. Church-Ellenberg-Farb proved the following equation using the Grothendieck-Lefschetz fixed point formula.  Given any $S_n$-representation $W$ with character $\chi_W$,
\begin{equation}\label{GL tori}
\sum_{T\in\tq}\chi_W(\s_T) = {q^{n(n-1)}}\sum_{i=0}^{n(n-1)/2} \dim H^{2i}(\tc;W)q^{-i}
\end{equation}
This formula was stated in Theorem 5.3 in \cite{CEF}. By additivity, the same formula holds when $W$ is taken to be a virtual representation of $S_n$.

\subsection{Arithmetic statistics for $F$-stable maximal tori in $\glq$}

In this subsection, we will compute the left-hand-side of (\ref{GL tori}) when $W$ is given by a character polynomial of the form  $X\choose\ld$. Our approach will be a systematic extension of Fulman's method in \cite{Fulman}. All the ideas in this subsection were already in Fulman's paper.

\begin{prop}
\label{count max tori}
For each fixed sequence of nonnegative integers $\ld=(\ld_1,\cdots,\ld_l)$, let $z_\ld:=\prod_{k=1}^l \ld_k!k^{\ld_k}$. We have the following equation of formal power series in $t$.
\begin{equation}
\sum_{n=0}^{\infty} \Bigg[\sum_{T\in \tq} \dbinom{X}{\ld}(\s_T) \Bigg] \frac{t^n}{|\glq|}= \frac{1}{z_\ld}\bigg[\prod_{k=1}^l \bigg(\frac{q^{-k}t^k}{1-q^{-k}}\bigg)^{\ld_k}\bigg]\cdot\bigg[\prod_{i=1}^\infty\frac{1}{1-q^{-k}t}\bigg]
\end{equation}
\end{prop}

\begin{proof}
We will use the following result of Fulman (stated as Theorem 3.2 in \cite{Fulman}).
\begin{thm*}[\textbf{Fulman}] With the notation as above,
\begin{equation}
\label{gen prod} \sum_{n=0}^\infty\bigg[\sum_{T\in\tq}\prod_{i=1}^nx_i^{X_i(\s_T)}\bigg]\frac{t^n}{|\glq|}= \prod_{k=1}^\infty \exp{\bigg[\frac{x_k t^k}{(q^k-1)k}\bigg]}
\end{equation}
\end{thm*}
Let $F(\vec{x},t)$ denote both sides of (\ref{gen prod}) as a  formal power series in infinitely many  variables $t$ and $x_1,x_2,\cdots$. We apply the formal differential operator $$(\frac{\partial}{\partial x})^{\ld}:=(\frac{\partial}{\partial x_1})^{\ld_1}(\frac{\partial}{\partial x_2})^{\ld_2}\cdots(\frac{\partial}{\partial x_l})^{\ld_l}$$ to the series $F(\vec{x},t)$ and then evaluate at $x_i=1$ for all $i$. Let $\ld!$ be an abbreviation for $(\ld_1!)(\ld_2!)\cdots(\ld_l!)$. Then
\begin{align*}
\ld!\sum_{n=0}^{\infty} \Bigg[\sum_{T\in \tq} \dbinom{X}{\ld}(\s_T) \Bigg] &\frac{t^n}{|\glq|}=(\frac{\partial}{\partial x})^{\ld}\bigg[F(\vec{x},t)\bigg]_{x_i=1, \forall i} \\
&=(\frac{\partial}{\partial x})^{\ld}\bigg[\prod_{k=1}^\infty \exp{\frac{x_k t^k}{(q^k-1)k}}\bigg]_{x_i=1, \forall i}\\
&=\bigg[\prod_{k=1}^l \bigg(\frac{t^k}{(q^k-1)k}\bigg)^{\ld_k}\bigg]\cdot\bigg[\prod_{k=1}^\infty \exp{\frac{t^k}{(q^k-1)k}}\bigg]\\
&=\bigg[\prod_{k=1}^l \bigg(\frac{t^k}{(q^k-1)k}\bigg)^{\ld_k}\bigg]\cdot\bigg[\prod_{i=1}^\infty \frac{1}{1-q^{-i}t}\bigg]
\end{align*}
where the last equality follows from 
$$\prod_{k=1}^\infty \exp{\frac{t^k}{(q^k-1)k}}=\prod_{i=1}^\infty \frac{1}{1-q^{-i}t}$$
which can be proved by expanding both sides into power series.  
\end{proof}

\subsection{Proof of Theorem \ref{intro main}, (II)}

For each $i$ and $n$, we abbreviate the twisted Betti number as 
\begin{equation}
\label{beta_i(n)}
\beta_i(n):=\dim H^{2i}(\tc;{X\choose \ld})
\end{equation}
 Define a formal power series in $z$ and $t$
\begin{equation}
\label{Psi}
\Psi_\ld(z,t):=\sum_{n=0}^\infty\sum_{i=0}^\infty \frac{\beta_i(n)}{(1-z)(1-z^2)\cdots(1-z^{n})}z^it^n
\end{equation}
We evaluate $\Psi_\ld(z,t)$ at $z=q^{-1}$:
\begin{align*}
\Psi_\ld(q^{-1},t)&= \sum_{n=0}^\infty\frac{1}{(1-q^{-1})(1-q^{-2})\cdots(1-q^{-n})}\sum_{i=0}^\infty {\beta_i(n)}q^{-i}t^n\\
&= \sum_{n=0}^\infty\frac{1}{(q^n-q^{n-1})(q^n-q^{n-2})\cdots(q^n-1)}\bigg[q^{n(n-1)}\sum_{i=0}^\infty {\beta_i(n)}q^{-i}(tq)^n\bigg]\\
&= \sum_{n=0}^\infty\frac{1}{|\glq|)}\sum_{n=0}^\infty\bigg[\sum_{T\in\tq}{X\choose\ld}(\s_T)\bigg](tq^{})^n\ \ \ \ \ \ \ \ \  \ \ \ \ \ \ \ \ \ \ \ \text{by \ref{GL tori}}\\
&=\frac{1}{z_\ld}\bigg[\prod_{k=1}^l \bigg(\frac{t^k}{1-q^{-k}}\bigg)^{\ld_k}\bigg]\cdot\bigg[\prod_{i=1}^\infty\frac{1}{1-q^{1-k}t}\bigg]\ \ \ \ \ \ \ \ \ \ \ \ \ \ \text{by Proposition \ref{count max tori}}
\end{align*}
Since the equality holds for all prime powers $q$, the equality also holds when $q^{-1}$ is replaced by a formal variable $z$ by Lemma \ref{cal 2}.

\qed

\subsection{Proof of Corollary \ref{intro cor}, (II).}
As before, it suffices to consider when $P={X\choose \ld}$. Let $\beta_i(n)$ be as in (\ref{beta_i(n)}) and let $\beta_i$ be $\lim_{n\to\infty}\beta_i(n)$, then we have 
\begin{equation}
\label{gen tori 1}
\lim_{n\to\infty}\sum_{i=0}^{n(n-1)}\frac{\beta_i(n)}{(1-z)(1-z^2)\cdots(1-z^n)}z^i= \sum_{i=0}^\infty \frac{\beta_i}{\prod_{j=1}^\infty (1-z^j)}z^i
\end{equation}

On the other hand, by Lemma \ref{Taylor}, we have 
\begin{align}
\nonumber\lim_{n\to\infty}\sum_{i=0}^{n(n-1)}\frac{\beta_i(n)}{(1-z)(1-z^2)\cdots(1-z^n)}z^i&=\bigg[(1-t)\Psi_\ld(z,t)\bigg]_{t=1}\\
\label{gen tori 2}&=\frac{1}{z_\ld}\prod_{k=1}^l\bigg(\frac{1}{1-z^k}\bigg)^{\ld_k}\cdot\prod_{j=1}^\infty\frac{1}{1-z^j}
\end{align}
Equating (\ref{gen tori 1}) and (\ref{gen tori 2}), we have 
$$\sum_{i=0}^\infty \beta_iz^i=\frac{1}{z_\ld}\prod_{k=1}^l\bigg(\frac{1}{1-z^k}\bigg)^{\ld_k}.$$
The generating function for $\beta_i$ is a rational function in $z$ with denominator  a polynomial of degree $|\ld|$. Thus $\beta_i$ satisfies a linear recurrence relation of length $|\ld|$.

\qed

\end{document}